\documentclass[12pt]{amsart}
\usepackage{amssymb,amscd,amsthm, verbatim,amsmath,color,fancyhdr, mathrsfs, bm, mathtools,soul,verbatim,listings}
\usepackage{graphicx}
\usepackage{turnstile}

\usepackage{tikz,tikz-cd}
\usepackage{array}
\usepackage{multirow}
\usepackage[normalem]{ulem}
\usepackage{amsaddr}
\usepackage{epsfig,psfrag}
\usepackage{amsfonts, amssymb}
\usepackage{enumerate}
\usepackage[all]{xy}
\usepackage[linesnumbered,ruled]{algorithm2e}
\usepackage{setspace}
\usepackage{float}
\usepackage{amsfonts, amssymb}
\usepackage{amssymb,epsfig,psfrag}
\usepackage{here}
\usepackage{pict2e}
\usepackage{hyperref}
\usepackage{setspace}
\usepackage{url}
\usepackage{here}
\usepackage{pict2e}
\usepackage{hyperref}
\usepackage{blindtext}
\usepackage{mathpazo}
\def\multiset#1#2{\ensuremath{\left(\kern-.3em\left(\genfrac{}{}{0pt}{}{#1}{#2}\right)\kern-.3em\right)}}

\hypersetup{
    colorlinks=true,
    linkcolor=blue,
    filecolor=magenta,      
    urlcolor=cyan,
}

\usepackage[letterpaper, left=2.5cm, right=2.5cm, top=2.5cm,
bottom=2.5cm,dvips]{geometry}
\newtheorem{theorem}{Theorem}
\newtheorem{lemma}[theorem]{Lemma}
\newtheorem{corollary}[theorem]{Corollary}
\newtheorem{proposition}[theorem]{Proposition}

\theoremstyle{definition}
\newtheorem{definition}{Definition}
\newtheorem{example}[definition]{Example}

\newtheorem{remark}[definition]{Remark}
\newtheorem{question}[definition]{Question}

\DeclareMathOperator{\Gal}{Gal}

\DeclareMathOperator{\Aut}{Aut}

\DeclareMathOperator{\disc}{disc}
\DeclareMathOperator{\res}{res}

\DeclareMathOperator{\sgn}{sgn}

\newcommand{\C}{\mathbb{C}}
\newcommand{\T}{\mathbb{T}}

\newcommand{\Proj}{\mathbb{P}}

\newcommand{\w}{\omega}
\newcommand{\z}{\zeta}

\newcommand{\la}{\langle}
\newcommand{\ra}{\rangle}

\newcommand{\inverselimitb}

\DeclarePairedDelimiter\floor{\lfloor}{\rfloor}

\title{A Unique Chief Series in the arboreal Galois Group of Belyi Maps}
\author{Wayne Peng}
\address{Department of Mathematics,
        University of Rochester
        Rochester, NY 14627}
\email{jpeng4@ur.rochester.edu}
\date{September 2020}

\begin{document}

\maketitle
\begin{abstract}
    We give a complete description of the normal subgroups of arboreal Galois groups of Belyi maps. The normal groups form a unique chief series. We also carefully compute the discriminate of the iterate of a polynomial minus an algebraic number, which allows us to predict when a such discriminate is a perfect square in the base field or intermediate field for a postcritically finite polynomials (PCF). As a consequence we are able to find another PCF cubic polynomial that has the same arboreal Galois group as the one of Belyi maps.
\end{abstract}
\section{Introduction}
Let $K$ be a number field with $\alpha\in K$, and let $f\in K[z]$ be a polynomial of degree $d\geq 3$. The Galois group of 
\[
f^n(z)-\alpha
\]
where $f^n$ is the $n$-th iterate of $f$. Such a Galois group is called arboreal Galois group or $n$-th dynamical Galois group, denoted by $\Gal_f^n(\alpha)=\Gal(f^n(z)-\alpha/K)$. If $\alpha$ is not an image of $\alpha$ and branch points of $f$ under $f^n$ for some positive integer $n$, then one can made a Galois action an act on a $d$-ary tree. The vertices of the $d$-ary tree is made up of the inverse image of $\alpha$ under $f^{-n}$ for $n=1,2,\ldots$, and where we draw an edge on two vertices $x$ and $y$ if $f(x)=y$. With a proper chosen $\alpha$, this tree is rooted at $\alpha$ with the Galois action acting faithfully on the tree. Sometime we call $\alpha$ the base point of the dynamical system of $f$.

The study of arboreal Galois group becomes popular for a while due to its application in number theory. Odoni \cite{Odoni1985} has showed that the description of the dynamical Galois group give rise to application on the density of primes division in certain dynamically defined sequences (see also \cite{Rafe2008,HJM2015,Looper2019} and \cite{Juul2019}). Generically the $n$-th dynamical Galois group of a degree $d$ polynomial is isomorphic to the group of automorphism of a $d$-ary tree $\Aut(T_n)$ which is isomorphic to the $n$-fold wreath product $[S_d]^n$ of symmetric group $S_d$ on $d$ letters (see \cite{juulthesis,2Odoni1985}). However for a specified chosen $f$ and $\alpha$ the dynamical Galois group can be a lot of smaller than the full wreath product. For example let $f$ be a powering map $x^d$ or a degree $d$ Chebyshev polynomial. These two examples are the simplest examples of \textbf{post-critical finite polynomial}, abbv. \textbf{PCF}. A polynomial or rational function $f$ PCF if all of its critical points $c$, i.e. $f'(c)=0$, is preperiodic. Jones, Pink and many other authors \cite{Rafe2013,pink2013profinite,GottesmanTang2010,BHL2017,RafeMichelle2014} shows that $\Gal_f^n(\alpha)$ has unbounded index inside $\Aut(T_n)$ as $n$ goes to infinity. The idea of the proof is realizing the $n$-th dynamical Galois group as a specialization of the arithmetic dynamical Galois group $\Gal(f^n-t/K(t))$, where $t$ is transcendental over $K$. The latter group is known as the arithmetic Galois group and and can be embedded in the profinite monodromy group $\pi_1^{\text{\'{e}t}}(\Proj^1_K\setminus P)$, where $P$ is strictly postcritical orbit.

The first case \cite{B-F-H-J-Y-2017} that has been explicit computed is the dynamical Galois group of $-2x^3+3x^2$ at a properly chosen base point $\alpha$. The polynomial is a degree $3$ Belyi map, a map that has three fixed critical points. They have shown that the $n$-th dynamical Galois group is isomorphic to $E_n^2$ (defined in~\ref{sec:2.1}). Then \cite{BEK2020} shows that the degree $3$ Belyi map is not a unique example. In fact all odd degree Belyi map with two exceptions are isomorphic to groups that has structure similar to $E_n^2$. Unlike the technique using in \cite{BEK2020} the first paper use pure algebraic argument to show the isomorphism. Following their observation we shows the following.
\begin{theorem}
Let $K$ be a number field. Let $f(z)=2z^3-3z^2+1$, and let $\alpha\in K$. Suppose there exist primes $\mathrm{p}$ and $\mathrm{q}$ of $K$ lying above $2$ and $3$ such that either $v_\mathrm{q}(\alpha)=1$ or $v_\mathrm{q}(1-\alpha)=1$, and either $v_\mathrm{p}(\alpha)=1$ or $v_\mathrm{p}(1-\alpha)=1$. Then for each $n\geq 1$,
\begin{enumerate}
    \item The polynomial $f^n$ is irreducible over $K$.
    \item The $n$-th dynamical Galois group $\Gal_f^n(\alpha)$ is isomorphic to $E_n^2$.
\end{enumerate}
\end{theorem}
Our key observation is using Equation~\ref{eq:disc} to find when the discriminant of $f^n-\alpha$ is a perfect square in the base field or an intermediate field. Such an observation give us a universal embedding of dynamical Galois group of PCF map.
\begin{theorem}
Let $f$ be a PCF polynomial over a number field $K$, and suppose $\alpha$ is not periodic. Let $\mathcal{C}_f$ be the set of all critical points of $f$. Let $L$ be the minimum integer such that $f^L(\mathcal{C}_f)$ is a periodic set, and let $O$ be the minimum positive integer such that $f^{L+O}(\mathcal{C}_f)=f^{L}(\mathcal{C}_f)$.
\begin{enumerate}
    \item Suppose the degree of $f$ is odd, and $L\leq 1$. Then $\Gal_f^n(\alpha)$ is a subgroup of $E_n^{2O}$.
    \item Otherwise $\Gal_f^n(\alpha)$ is a subgroup of $F_n^{(m_1,m_2)}$ where
    \[
    (m_1,m_2)=\begin{cases}
    (O+1,1),\quad L=0\text{ and }O\text{ is even;}\\
    (O+2,2),\quad L=0\text{ and }O\text{ is odd;}\\
    (L+2O-1,L-1),\quad L>1\text{ and }d\text{ is odd;}\\
    (L+O,L),\quad \quad L>1\text{, $d$ is odd, and $a_f$ is a perfect square or $O$ are even;}\\
    (L+O+1,L+O),\quad L>1\text{, $d$ is odd and $a_f$ is not a perfect square and $O$ is odd.}
    \end{cases}
    \]
\end{enumerate}
\end{theorem}
As a consequence, the index of dynamical Galois group of PCF map is unbounded in $\Aut(T_n)$ as $n$ goes infinite, follows immediately after this theorem. This provide a pure algebraic proof of this important theorem.

Due to $E_n^2$ represents the dynamical Galois group of a large family of polynomials, it makes the author curious about the structure of such groups. One of the purpose of this paper is to provide a pure algebraic point of views on proving that the rank of $E_n^2$ is bounded as $n$ goes to infinity. The rank of a group $G$, denoted by $d(G)$, is the cardinality of minimal generating set the group. It is well-known that $d(G)\geq d(G/N)$ where $N$ is a normal subgroup of $G$. Jones uses this lower bound and shows that the abelianization of $[S_d]^n$ is the $n$-fold direct product of $C_2$, so the arithmetic Galois groups of PCF maps, which is finite generated group, cannot be of finite index in $[S_d]^n$. 

To obtain an upper bound of $d(G)$ by subgroups or quotient groups of $G$ is difficult. In~\cite{DetomiLucchini2003} the authors have shown that $d(G)=d(G/N)$ if and only if $G/N$ has not too many factors $G$-equivalent to $N$. We don't need such a strong theorem. Instead \cite{DFLA1998} shows that $d(G)=\max\{2,d(G/N)\}$ if $G$ is not cyclic and $N$ is the unique minimum subgroup of $G$. Thus if we can find a unique minimum subgroup $N$ in $E_n^2$ such that we can again find a unique minimum subgroup in the quotient $E_n^2/N$ and if this process can be repeated until the quotient subgroup is small enough to compute, then we are able to bound the rank of $E_n^2$. The similar concept of the above process is chief series of a group $G$, which is a sequence of normal subgroups $N_i$ of $G$ satisfying
\[
1=N_1\triangleleft N_2\triangleleft\cdots\triangleleft N_k=G
\]
and each chief factor $N_{i+1}/N_{i}$ is a minimum normal subgroup of the quotient group $G/N_i$ (see~\cite{Isaacs2009} for details). Our result, in Section~\ref{sec:the rank of E_n}, is the following.
\begin{theorem}
Let $d$ be odd. The rank of the group $E_n^2(d)$ is $2$ for all $n$. Moreover $E_n^2$ has a unique chief series. 
\end{theorem}

The outline of this paper is as follows. In Section~\ref{sec:2} we define the group $E_n^m$ and $F^{(m_1,m_2)}_n$, give a brief on the relation between the abelianization and wreath product, and explicitly compute the discriminate of $f^n-\alpha$. Most stuff in Section~\ref{sec:2} is already well-understood. In Section~\ref{sec:3} we give our main results. Finally, in Section~\ref{sec:3}, the author has some open questions and the difficulty to generalize the argument of this paper to more general cases.

We will use the following symbols in this paper.
\begin{itemize}
    \item $1$ is use to denote either the integer $1$, an identity map or an identity of a group. If we use $1$ to denote an identity, the group will be specified in the context.
    \item We will use $\ast$ to denote an arbitrary element. Where the element belongs to will be specified in the context.
    \item For $\sigma,\tau\in G$ we denote $\sigma^\tau$ for the conjugacy $\tau\sigma\tau^{-1}$.
    \item Let $\mathcal{C}_f$ be the set of all critical points of $f$, and $f(\mathcal{C}_f)=\{f(c)\mid c\in\mathcal{C}_f\}$.
    \item The splitting field of $f^n-\alpha$ is denoted by $K_f^n(\alpha)$, and the correspondent Galois group is denoted by $\Gal_f^n(\alpha)$.
    \item A point or an object $\alpha$ is preperiodic if $\{f^n(\alpha)\mid n=1,2,\ldots\}$ is a finite set. If we say it is periodic, then $f^n(\alpha)=\alpha$ for some integer $n$. The tail length of a preperiodic point is the smallest integer such that $f^n(\alpha)$ becomes a periodic point.
\end{itemize}
\section{Preliminary}\label{sec:2}
\subsection{Wreath product and arboreal representation}\label{sec:2.1}

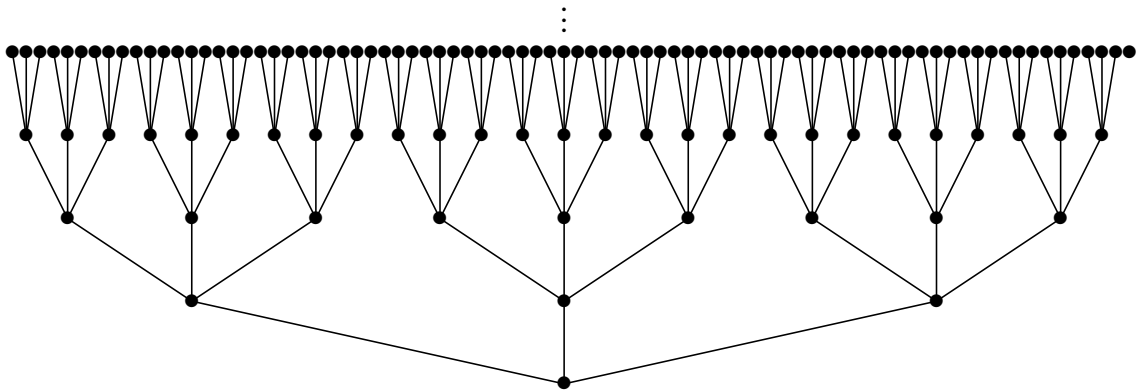
\begin{figure}[h]\label{3-ary tree}
\centering
\begin{tikzpicture}
[scale = 0.55]
	\node (top) at (8.5,9) {$\vdots$};
	\node (0) at (8.5,0) {$\color{black}\bullet$};
	\foreach \x in {40,41,...,121}
		\node (\x) at (1/3*\x -80/3+8.5,8) [inner sep=0.5pt] {$\color{black}\bullet$};
	\foreach \x in {13,...,39}
		\node (\x) at (\x - 26 +8.5,6) [inner sep=3pt] {$\color{black}\bullet$};
	\foreach \x in {4,5,...,12}
		\node (\x) at (3*\x-24 +8.5,4) [inner sep=3pt] {$\color{black}\bullet$};
	\foreach \x in {1,2,3}
		\node(\x) at (9*\x - 18+8.5,2) [inner sep=3pt]  {$\color{black}\bullet$};
	\foreach \x [evaluate=\x as \y using int(3*\x + 1),evaluate=\x as \z using int(3*\x + 2),evaluate=\x as \w using int(3*\x + 3)] in {0,1,...,39}{
		\draw[line width=.6pt] (\x.center) -- (\y.center);
		\draw[line width=.6pt] (\x.center) -- (\z.center);
		\draw[line width=.6pt] (\x.center) -- (\w.center);
	}
\end{tikzpicture}
\caption{A 3-ary tree of 4 levels}
\end{figure}
\begin{comment}
\begin{tikzpicture}
[scale = 0.55]
%\node at (0.2,0.2) {$\color{white}\bullet$};
\node at (0.75,0.2) {$\alpha$};
\node at (0.75,8.2) {$f^{-4}(\alpha)$};
\node at (0.75,6.2) {$f^{-3}(\alpha)$};
\node at (0.75,4.2) {$f^{-2}(\alpha)$};
\node at (0.5,2.2) {$\;\;\;f^{-1}(\alpha)$};
\end{tikzpicture}
\centering
    \includegraphics{}
    \caption{Caption}
    \label{fig:my_label}
\end{figure}
\end{comment}

Let $T_n(d)$ be a regular $d$-ary tree of $n$ levels (see Figure~\ref{3-ary tree} as an example). We will omit $d$ in the notation for simplicity sake. $T_n$ has $d^n$ many leaves, vertices on the top of the tree, and has $1+d+\cdots +d^n$ many vertices on the tree. The \textbf{level of a vertex on a tree} is distance between the vertex and the root. For $m\leq n$ the \textbf{$m$-th level of the tree $T_d$} is the set of all vertices of level $m$.

Our result and many arguments will depend on an explicit labeling of the tree. We will make this labeling explicitly for the purpose of rigor, but we won't commend on it in the rest of this paper.
There are two labeling system of a tree. The \textbf{regular labeling} to denote a vertex of level $m$ is to use a tuple $(l_m,\ldots, l_1)$ in $\{1,2,\ldots,d\}^m$. Note that the vertex with the label $(l_m,\ldots, l_1)$ has distance $1$ to the vertex with the label $(l_{m-1},\ldots, l_1)$. This gives a recursive relation to label all vertices of a tree.

We are able to identify some canonical subtree using the regular labeling. For example, for each tuple $(l_{m},\ldots,l_{1})$ with $m\leq n$, the set of vertex $(\ast,\cdots,\ast,l_{m},\ldots,l_1)\in T_n$ is a subtree isomorphic to $T_{n-m}$ rooting at $(l_m,\ldots, l_1)$. We will call this subtree a level $n-m$ subtree, and denote by $T_{(l_m,\ldots, l_1)}$. 

The above system of labeling is not convenient sometime, especially when we only talk about the $m$-th level of the tree. For each vertex of level $m$ we label the point by the integer $l_m+(l_{m-1}-1)d+(l_{m-2}-1)d^2+\cdots +(l_1)d^{m-1}$. We will denote this labeling system by $I(T_n)=\{1,2,\cdots d^{m}\}$. The tree structure splits $I(T_n)$ to different blocks. A \textbf{block} of $T_n$ over $T_m$, denote by $\mathcal{B}(T_n/T_m)(l_{n-m},\ldots, l_1)$, are the set of indices that belongs to the same level $n-m$ subtree, and we use $\mathcal{B}(T_n/T_m)$ be the set of all such blocks. An index $i$ belongs to a level $n-m$ subtree $T_{(l_m,\ldots, l_1)}$ if $i\in\mathcal{B}(T_n/T_m)(l_m,\ldots, l_1)$. 

We will omit the strict definition of $\Aut(T_n)$ and how the group acting on $T_n$. We refer readers to \cite{B-F-H-J-Y-2017} for details. The labeling systems of the tree $T_n$ induce an isomorphism $[S_d]^n\subseteq S_{\# I(T_n)}$ where $[S_d]^n$ is the $n$-folds wreath product of $S_d$. On the other hand the embedding to the symmetric group $S_{d^n}$ of degree $d^n$ allows us to use all terminologies adhering to the theory of symmetric groups. For example an index $i\in I(T_n)$ is in the \textbf{support of $\sigma\in\Aut(T_n)$} if $\sigma(i)\neq i$. We also mimic the definition of $\sgn$ in \cite{B-F-H-J-Y-2017}, and use it to construct the group $E_n^m$ and $F_n^{(m_1,m_2)}$. We will give the details later.

Let $G\subseteq S_d$ be a group acting on $I=\{1,2,\ldots, d\}$. For an arbitrary group $H$ the wreath product of $H$ by $G$ over $I$, denoted by $H\wr_{I} G$, is an extension of groups. As a set the element of $H\wr_{I}G$ can be expressed as
\[
((h_1,\ldots, h_d);g)
\]
with $h_i\in H$ and $g\in G$. The operation of two element is defined as
\[
((h_i)_{i\in I}; g)((h_i')_{i\in I}; g')=((h_g(i))_{i\in I}(h_i');gg')
\]
Note that the following simple decomposition of an element of wreath product 
\[
((h_1,\ldots, h_d);g)=((1,\ldots, 1);g)
((h_1,\ldots, h_d);1)
\]
will be use frequently in our argument. When the index set $I$ is obvious in the context, we will omit $I$. A useful observation is that, for $n>m$, $\Aut(T_n)\cong \Aut(T_m)\wr_{I(T_{n-m})}\Aut(T_{n-m})$. If we say $\sigma\in\Aut(T_n)$ or $\sigma\in G\subseteq \Aut(T_m)\wr\Aut(T_{n-m})$, it means we will express $\sigma$ as
\[
((a_i)_{i\in I(T_{n-m})};b)
\]
with $a_i\in \Aut(T_m)$ and $b\in\Aut(T_{n-m})$. There is a natural homomorphism $\res:H\wr G\to G$ by projecting the second coordinate to $G$.

Let me recall that the sign of a permutation is $1$, if the permutation is even, or $-1$ otherwise. One can define a sign on the automorphism of a $d$-ary tree by labeling the branches on the $n$-th level of the tree. Once a label system is given, it gives an embedding from $\Aut(T_n)$ to $S_{d^n}$. Compose with the sign of permutations, it defines a sign on $\Aut(T_n)$. One may give a different label system on the branches, but two different sets of labels are conjugate by an element in $S_{d^n}$ which doesn't change the sign, so we have a unique way to define $\sgn:\Aut(T_n)\to C_2=\{\pm 1\}$. In addition there is a natural homomorphism, denoted by $\res_{m}$ from $\Aut(T_n)$ to $\Aut(\T_m)$ by restriction, and we define 
\[
\sgn_{m}=\sgn\circ \res_{m}.
\]

Here is a proposition about the sign on the automorphism of a $d$-ary tree.
\begin{proposition}\label{prop:sgn}
Let $T_n$ be a $d$-ary tree to $n$ levels. For odd $d$ let $\sigma=((b_i)_{i\in I};a)\in \Aut(T_n)$ where $a\in \Aut(T_m)$, $I=\{1,2,\ldots, d^m\}$ and $b\in \Aut(T_{n-m})$ for some $m<n$. Then we have
\[
\sgn(\sigma)=\sgn(a)\prod_{i\in I}\sgn(b_i).
\]
For an even $d$ we write $\sigma$ as $((b_i)_{i\in I}, a)\in \Aut(T_1)\wr\Aut(T_{n-1})$, and so
\[
\sgn(\sigma)=\prod_{i\in I}\sgn(b_i)
\]
\end{proposition}
\begin{proof}
First, note that $\sgn(((b_i)_{i\in I},a))=\sgn(((1)_{i\in I}, a)))\sgn(((b_i)_{i\in I},1)))$. Since $b_i$ acts disjointedly on the tree $T_n$, it is clear $\sgn(((b_i)_{i\in I},1))=\prod_{i\in I}\sgn(b_i)$. Now consider $\sgn(((1)_{i\in I}, a))$. Let me express $a$ in terms of product of disjoint cycle $a_1\cdots a_k$ for some integer $k$. Each cycle $a_i$ creates $d^{n-m}$ many disjoint cycles on the $n$-th level of the tree $T_n$, the cycles associating with $a_i$ are disjointing from the cycles associating with $a_j$ for distinct $i$ and $j$. Therefore the sign of $((1)_{i\in I}, a)$ is
\[
\prod_{i=1}^k\sgn(a_i)^{d^{m-n}}=
\begin{cases}
1 \text{, if }d\text{ is even;}\\
\prod_{i=1}^k\sgn(a_i)\text{, o.w..}
\end{cases}
\]
Note that for even $d$ we can set $m=n-1$, and it implies $\sgn(\sigma)=\prod_{i\in I}\sgn(b_i)$.
\end{proof}

Note that if we consider $\Aut(T_n)\cong\Aut(T_{n-m})\wr \Aut(T_{m})$ with $0<m<n$, it is natural to define $\sgn^m:\Aut(T_n)\to \{\pm 1\}$ as
\[
\sgn^m((b_i)_{i\in I}, a_i)=\prod_{i\in I}\sgn(b_i)
\]
with $b_i\in \Aut(T_{n-m})$ and $a\in\Aut(T_{m})$. Obviously this is a group homomorphism.

Let $n>m_1>m_2$. we can also define an alternative sign function as
\begin{equation}\label{eq:sgn2}
\sgn^{(m_1,m_2)}=\sgn^{m_2}\circ \res_{m_1}
\end{equation}
\begin{proposition}\label{proposition:sgn2}
We have
\[
\sgn^{(m_1,m_2)}=\sgn_{m_1}\sgn_{m_2}.
\]
\end{proposition}
\begin{proof}
It follows from direct computation. A $\sigma\in\Aut(T_n)$ with $n\geq m$ can be expressed as $((b_i),a)\in\Aut(T_{n-m})\wr\Aut(T_m)$, and note that $\sgn_{m}(\sigma)=\sgn(a)$. Thus we have
\[
\sgn(\sigma)\sgn_{m}(\sigma)=\sgn(a)\prod_{i\in I}\sgn(b_i)\sgn_{m}(\sigma)=\prod_{i\in I}\sgn(b_i)=\sgn^m(\sigma).
\]
Thus by Equation~\ref{eq:sgn2} we have
\begin{align*}
    \sgn^{(m_1,m_2)}&=\sgn^{m_2}\circ\res_{m_1}\\
    &=(\sgn\cdot\sgn_{m_2})\circ\res_{m_1}\\
    &=(\sgn\circ\res_{m_1})\cdot(\sgn_{m_2}\circ\res_{m_1})\\
    &=\sgn_{m_1}\cdot(\sgn\circ\res_{m_2}\circ\res_{m_1})=\sgn_{m_1}\cdot\sgn_{m_2}.
\end{align*}
\end{proof}
Now we define the following two subgroups of $\Aut(T_n)$
\[
E_n^m=
\begin{cases}
\Aut(T_1)\text{, if }n=1;\\
E_{n-1}^m\wr \Aut(T_1)\text{, if }n<m;\\
E_{n-1}^m\wr \Aut(T_1)\cap \ker(\sgn_m)\text{, if }n\geq m.
\end{cases},
\]
and
\[
F_n^{(m_1,m_2)}=
\begin{cases}
\Aut(T_1)\text{, if }n=1;\\
F_{n-1}^{(m_1,m_2)}\wr \Aut(T_1)\text{, if }n< m_1;\\
F_{n-1}^{(m_1,m_2)}\wr \Aut(T_1)\cap \ker(\sgn^{(m_1,m_2)})\text{, if }n\geq m_1.
\end{cases}.
\]
\begin{example}
Note that $E_n^m$ is not isomorphic to $F_n^m$. For example for $d=3$ the element $(((12),(12),1),(12))$ is in $F_2^{(2,1)}$ but not in $E_2^2$. 
\end{example}
Let us compute the order of these groups.
\begin{proposition}\label{prop:unboundedindex}
Let $E_n^m$ and $F_n^m$ be subgroups of the automorphism of a $d$-ary tree with an odd $d$. Then
\[
|E_n^m|=|F_n^m|=
\dfrac{(d!)^{\frac{d^n-1}{d-1}}}{2^{\frac{C(d,n)}{d-1}}}
\]
with
\[
C(d,n)=\begin{cases}
0\text{, if }n<m\\
\dfrac{d^{n-m+1}-1}{d-1}\text{, o.w.}
\end{cases}.
\]
\end{proposition}
\begin{proof}
For $n\geq m$ we consider the following two compositions
\[
\begin{tikzcd}
\phi_1: E_{n-1}^m\wr \Aut(T_1)\arrow[r,hook]& \Aut(T_n) \arrow[r,"\sgn_m"]&\{\pm 1\}
\end{tikzcd}
\]
and
\[
\begin{tikzcd}
\phi_2: F_{n-1}^{(m,m')}\wr \Aut(T_1)\arrow[r,hook]& \Aut(T_n) \arrow[r,"\sgn^{(m,m')}"]&\{\pm 1\}
\end{tikzcd}.
\]
We claim that $\phi_1$ and $\phi_2$ are onto.

Take $g=((1)_{i\in I},\sigma)\in E_{n-1}\wr \Aut(T_1)$ for odd permutation $\sigma$, and one computes $\sgn_m(g)=sgn(\sigma)=-1$ by assuming $d$ is odd. 

For $\phi_2$ let us first prove the following assertion: for any $m\geq 1$ if $\sgn_{m}:F_{n-1}^{(m_1,m_2)}\to C_2$ is onto, then so is $\sgn_{m+1}:F_{n-1}^{(m_1,m_2)}\wr\Aut(T_1)\to C_2$. Since $\sgn_m$ is onto, we may choose an element $a$ such that $\sgn_m(a)=-1$. Then $b=((a_i)_{i\in I(T_1)};1)$ with $a_1=a$ and $a_i=1$ for $i\geq 2$ has $\sgn_{m+1}(b)=-1$. Now we define a homomorphism, for $n\geq m_1-1$,
\begin{align*}
\phi:F_{n-1}^{(m_1,m_2)}\wr\Aut(T_1)&\to C_2^{m_1}\\
a&\mapsto(\sgn_{m_1}(a),\sgn_{m_2}(a),\ldots,\sgn_{1}(a)).
\end{align*}
It is clear that $\phi$ is onto as $n=m_1$ because the domain is a wreath product of $S_d$, so $F_{m_1}^{(m_1,m_2)}$ is the kernel of $\phi_2=\psi\circ\phi$ where
\begin{align*}
\psi: C_2^{m_1}&\to C_2\\
(c_{m_1},\ldots,c_{1})&\mapsto c_{m_1}c_{m_2} 
\end{align*}
by Proposition~\ref{proposition:sgn2}. Note that $\sgn_{m}(F_{m_1}^{(m_1,m_2)})$ is onto $C_2$ for all $m$, so $\phi$ is onto for $m_1+1$. Thus it shows that $\phi_2:F_{m_1+1}^{(m_1,m_2)}\to C_2$ is onto. Thus by induction we may show that $\phi_2$ is onto for any $n\geq m_1$.

Finally we calculate the orders of these groups. It is obviously equal to $(d!)^{\frac{d^n-1}{d-1}}$ for all $n<m$ since $\Aut(T_n)\cong [S_d]^n$. Now we can prove the formula by induction. Suppose the formula is true for all $n$ with $n\geq m-1$, then
\[
|E_n^m|=\dfrac{1}{2}|E_n^m\wr\Aut(T_1)|=\dfrac{1}{2}\left(\dfrac{(d!)^{\frac{d^n-1}{d-1}}}{2^{\frac{d^{n-m+1}-1}{d-1}}}\right)^dd!=\left( \dfrac{(d!)^{\frac{d^{n+1}-1}{d-1}}}{2^{\frac{d^{n-m+2}-1}{d-1}}}\right).
\]
Similarly the same formula holds for $F_n^{m_1,m_2}$ by the same reasoning.
\end{proof}
\begin{lemma}\label{lemma:transitive}
For $d>2$ the $E_n^m$ acts transitively on the correspondent index set. Moreover for two distinct index there exists $\sigma\in E_n^m$ permuting $i$ and $j$.
\end{lemma}
\begin{proof}
If we prove the second statement, then the first statement is a consequence. We process the proof of the second statement by induction. Moreover we notice that $E_n^1\subseteq E_n^m$ for any $m$, so we only need to show the second statement is true for $E_n^1$ for any $n>0$.

It is trivial for $E_1^1$ because it is isomorphic to $A_d$. For induction hypothesis we assume the statement is true for some $n$. For $i$ and $j$ on the same level $k$ subtree with $k<n$, the statement is true by induction. If $i$ and $j$ does not contain in the same level $k$ subtree with $k<n$, then we may choose a transpose $\tau$ that permutes the subtree of $i$ and the subtree of $j$ and choose a permutation $\sigma_i$ on the subtree of $i$ that taking $\tau(j)$ to $i$ and choose a permutation $\sigma_j$ on the subtree of $j$ that taking $\tau(i)$ to $j$. Then the permutation 
\[
    \omega=((1,\ldots, 1,\sigma_i,1\ldots,1,\sigma_j,1,\ldots,1),\tau\tau')
\]
where $\sigma_i$ and $\sigma_j$ are at the coordinates correspondent to the subtree of $i$ and $j$ respectively, and $\tau'$ is a transpose disjoint from $\tau$.
\end{proof}
\begin{comment}
The last group structure that we will use later in this paper is the outer semidirect product $C_2^{d^n}\ltimes \Aut(T_n)$. Given $\sigma\in \Aut(T_n)$ $\sigma$ gives an automorphism of $C_2^{d^n}$ by permuting the coordinates. A version of splitting lemma of groups gives the following exact sequence
\[
1\to C_2^{d^n}\to C_2^{d^n}\ltimes \Aut(T_n)\xrightarrow[\phi]{} \Aut(T_n)\to 1.
\]
with an isomorphism $\eta:\Aut(T_n)\to C_2^{d^n}\ltimes \Aut(T_n)$ such that $\eta\circ \phi=1$. We have the following lemma, which play a key rule to analyze the structure of $E_n^{m}$ and $F_n^{m_1,m_2}$.
\end{comment}
\subsection{Abelianization}
The context in this section is already well-known. We give the statement and proof for the sake of self-contained. The primary goal of this section is to give the relation between the abelianization and the wreath product. It is an elementary exercise in the course of abstract algebra to show that if $G/N$ is abelian for a normal subgroup $N$ of a group $G$, then $N$ contains the commutator subgroup of $G$. Thus, one can define
\begin{definition}
The \textbf{abelianization} of a group $G$, denoted as $G^{ab}$, is the largest abelian quotient group of $G$. More precisely, let $H_1$ and $H_2$ be subgroups of $G$. Let me define $[H_1,H_2]$ to be the subgroup generated by the set $\{h_1h_2h_1^{-1}h_2^{-1}\mid  h_1\in H_1\ \text{and}\ h_2\in H_2\}$. The commutator subgroup of $G$ is $[G,G]$, and $G^{ab}$ is the quotient group $G/[G,G]$.
\end{definition}

Let $K$ be Galois over $L$. A quotient group of a Galois group $\Gal(K/L)$ is correspondent to a normal subextension of $L$ in $K$. Therefore, the abelianization of the Galois group is correspondent to the largest abelian and normal subextension of $L$ in $K$.

The following theorem deduce the relation between wreath product and abelianization. \begin{lemma}
Let $G$ be the semidirect product of $H$ acting on $N$. Then, $G^{ab}=(H\ltimes N)^{ab}=(H^{ab})\times (N^{ab})_H$ where $(N^{ab})_H$ is the quotient group $N^{ab}/[H,N]$.
\end{lemma}
\begin{proof}
Since $[G,G]=\la [N,N],[N,H],[H,H]\ra$, we have 
\begin{align*}
    G^{ab}&=\dfrac{G}{[G,G]}=\dfrac{G}{\la [N,N],[N,H],[H,H]\ra}\\
    &\cong \dfrac{H\ltimes (N^{ab})}{\la [H,N],[H,H]\ra}\cong \dfrac{H\times (N^{ab})_N}{\la [H,H]\ra}\\
    &\cong H^{ab}\times (N^{ab})_H.
\end{align*}
\end{proof}
Using this lemma, we can compute the abelianization of a wreath product.
\begin{lemma}\label{lemma:abelinaiizationofwreathproduct}
Let $H$ and $G$ be groups with $G$ acting on $G$ faithfully and transitively. Then, $(G\wr H)^{ab}=G^{ab}\times H^{ab}$.
\end{lemma}
\begin{proof}
As before, I denote $B=H^r$, and $G$ is a group acting on the index set of $r$ many elements faithfully, i.e., for any $i$, there is some $g\in G$ such that $g(i)\neq i$. The above lemma implies $G\wr H=G\ltimes B$, and so $(G\wr H)^{ab}=G^{ab}\times (B^{ab})_G$. Moreover, $B^{ab}=(H^r)^{ab}=(H^{ab})^r$, so it remains to show that $((H^{ab})^r)_G=H^{ab}$.

Let us define the following map, and show it is a surjective homomorphism with $[G,B]$ as the kernel. Let $\phi: (H^{ab})^r\to H^{ab}$ be
\[
\phi((h_i)_{i})=\sum_{i=1}^r h_i.
\]
Since $H^{ab}$ is abelian, the map is absolutely homomorphic. Beside we are able to choose $(h_i)$ such that $h_1$ is any element in $H^{ab}$ and $h_i=1$ for all $i\neq 1$, so $\phi$ is surjective.

It is clear that $\ker(\phi)\supseteq [(H^{ab})^r,G]$. Conversely, given $\phi((b_i))=\sum_{i=1}^rb_i=0$, one has
\[
b_1=-\sum_{i=2}^r b_i.
\]
Since $G$ acts on $H$ faithfully and transitively, for each $i$ there is a $g_i$ such that $g_i(i)=1$. For $i\neq 1$, let us define $b_{i_0}=(b_{i}')$ with $b_{i}'=1$ for all $i\neq i_0$ and $b'_{i_{0}}=b_{i_{0}}$. Thus, I have $b^{-1}_{i_0}g_{i_0}=(b''_{i})$ with $b_1=b'_{i_0}$ and $b_i=1$ for all $i=2,3,\ldots,3$ and
\[
(b_i)=\prod_{i_0=2}^rg^{-1}_{i_0}b^{-1}_{i_0}g_{i_0}b_{i_0},
\]
which shows that $(b_i)\in [(H^{ab})^r,G]$.
\end{proof}

\subsection{Discriminant}
\begin{lemma}\label{lemma:discriminant}
Let $a_f$ be the leading coefficient of a polynomial $f$ over any field $K$. The discriminant  
\begin{equation}\label{eq:disc}
\disc(f^n(x)-\alpha)=(-1)^{A(d,n)}a_f^{B(d,n)}d^{d^n}\disc(f^{n-1}(x)-\alpha)^{d}\prod_{c\in (f')^{-1}(0)}(f^n(c)-\alpha).
\end{equation}
where
\begin{align*}
    A(d,n) & =d^n(\dfrac{d^n-1}{2}+\dfrac{d^{n-1}-1}{2})\\
    B(d,n) & =d^{2n-1}-1
\end{align*}
for all $n\geq 1$.
In particular, we have
\[
(-1)^{A(d,n)}=\left(\dfrac{d}{4}\right)=\begin{cases}
1 & \text{ if } d\text{ is even or }\equiv 1\mod 4 \\ 
-1 & \text{o.w.} 
\end{cases}
\]
for any $n$.
\end{lemma}
\begin{proof}
By the definition of the determinant, we have
\begin{align*}
\disc(f^n(x)-\alpha) & =(-1)^{d^n(d^n-1)/2}a_{f^n}^{d^n-2}\prod_{a\in f^{-n}(\alpha)}(f^n)'(a)\\
& =(-1)^{d^n(d^n-1)/2}a_{f^n}^{d^n-2}\prod_{a\in f^{-n}(\alpha)}(f')(f^{n-1}(a))f'(f^{n-2}(a))\cdots f'(f(a))f'(a)
\end{align*}
where $d$ is the degree of $f$. In terms of $a_f$, one has
\[
a_{f^n}=a_f^{d^{n-1}+d^{n-2}+\cdots +1},
\]
so
\[
a_{f^n}^{d^n-2}=a_f^{\frac{(d^n-1)(d^n-2)}{d-1}}.
\]
Note that
\begin{align*}
\prod_{a\in f^{-n}(\alpha)}(f')(f^{n-1}(a))f'(f^{n-2}(a))\cdots f'(f(a))=\left(\prod_{a\in f^{-(n-1)}(\alpha)}(f')(f^{n-2}(a))f'(f^{n-3}(a))\cdots f'(a)\right)^{d}\\
=\left((-1)^{d^{n-1}(d^{n-1}-1)/2}a_{f^{n-1}}^{-(d^{n-1}-2)}\disc(f^{n-1}(x)-\alpha)\right)^d\\
=(-1)^{d^n(d^{n-1}-1)/2}a_f^{-d\frac{(d^{n-1}-1)(d^{n-1}-2)}{d-1}}\disc(f^{n-1}(x)-\alpha)^d,
\end{align*}
so we have
\[
\disc(f^{n}(x)-\alpha)=(-1)^{d^n(\frac{d^n-1}{2}+\frac{d^{n-1}-1}{2})}a_f^{d^{2n-1}-2}\disc(f^{n-1}(x)-\alpha)^{d}\prod_{a\in f^{-n}(\alpha)}f'(a).
\]
The product in the tail is equal to $\res(f^n(x)-\alpha,f'(x))/(a_{f^n})^{d-1}$, and we can rewrite the resultant in terms of the product of $f^n(c)$ where $c$ is a root of $f'$, a critical point of $f$. That is
\[
\prod_{a\in f^{-n}(\alpha)}f'(a)=a_{f^n}^{-(d-1)}\res(f^n(x)-\alpha,f'(x))=a_{f}^{-d^n+1}(-1)^{d^n(d-1)}(da_f)^{d^n}\prod_{c\in(f')^{-1}(0)}f^n(c)-\alpha.
\]
Note that $-1^{d^n(d-1)}=1$ for all positive integers $d$ and $n$. Putting everything together, we have
\[
\disc(f^n(x)-\alpha)=(-1)^{A(d,n)}a_f^{B(d,n)}d^{d^n}\disc(f^{n-1}(x)-\alpha)^{d}\prod_{c\in (f')^{-1}(0)}(f^n(c)-\alpha),
\]
where
\begin{align*}
    A(d,n) & =d^n(\dfrac{d^n-1}{2}+\dfrac{d^{n-1}-1}{2})\\
    B(d,n) & =d^{2n-1}-1.
\end{align*}
\end{proof}
\begin{remark}
Together Theorem 1 in \cite{GNT2013}, one can show that $\disc(f^n(x)-\alpha)$ has a primitive prime for all but finite many $n$ by assuming $K$ an $abc$-field. This is not surprising at all since Odoni in \cite{Odoni-1997} shows that for generic polynomial $f$ the dynamical Galois group is isomorphic to the full wreath product.
\end{remark}
\begin{lemma}\label{lemma:quadraticextension}
If $\Gal_f^n(\alpha)\cong [S_d]^n$, then $K(\{\sqrt{\disc(f^m)(x)-\alpha}\mid m=1,\ldots, n\})$ is of index $2^{n}$ over $K$. In particular, any quadratic subextension of $K_f^{n}(\alpha)$ is a subfield of $K(\{\sqrt{\disc(f^m(x)-\alpha)}\mid m=1,\ldots, n\})$.
\end{lemma}
\begin{proof}
By Lemma~\ref{lemma:abelinaiizationofwreathproduct}, the abelianization $[S_d]^{ab}$ is the $n$ fold direct product of $S_d^{ab}=C_2$. Therefore, if we can find an abelian extension $K^{ab}$ of $K$ contained in $K_f^n(\alpha)$ with the index $[K^{ab}:K]=2^n$ and it is built from adjoining square roots, we find all possible quadratic extensions of $K$ in $K_f^{n}(\alpha)$.

It remains to show that $K(\{\sqrt{\disc(f^m(x)-\alpha)}\mid m=1,\ldots,n\})$ is a subfield of $K_f^n(\alpha)$ that is of index $2^n$ over $K$. We will show the following two assertions. We will first show that $\sqrt{\disc(f^m(x)-\alpha)}$ is not in $K$ for any $m=1,\ldots,n$. Next, we will show that $K(\sqrt{\disc(f^{m'}(x)-\alpha)})$ is not in $K(\{\sqrt{\disc(f^m(x)-\alpha)}\mid m=1,\ldots,m'-1\})$ for all $m'=1,\ldots,n-1$.

Since the Galois group is isomorphic to $[S_d]^n$, $\Gal_f^m(\alpha)$ is isomorphic to $[S_d]^m$ for all $m=1,\ldots, n$. Moreover for any distinct branches $a,a'\in f^{-m}(\alpha)$, there is a Galois map $\sigma\in \Gal_f^m(\alpha)$ switching $a$ and $a'$ and fixing other roots. Indexing the roots of $f^{m}(x)-\alpha=0$ with integers, denoted as $a_i$, we have
\[
\prod_{i<j}(a_i-a_j)\xmapsto{\sigma} -\prod_{i<j}(a_i-a_j).
\]
Hence, we know $\sqrt{\disc(f^{m}(x)-\alpha)}\not\in K$.

For the second part, we use the fact that $\Gal_f^n(\alpha)\cong[S_d]^n$ again. It implies we can find a Galois map $\sigma'$ only switching a pair of branches $a,a'\in f^{-m}(\alpha)$ and $f(a)=f(a')$. It follows that $\sigma'$ fixes $\sqrt{\disc(f^{m'}(x)-\alpha)}$ for all $m'<m$, but does not fix $\sqrt{\disc(f^{m}(x)-\alpha)}$, which is the desired result. 
\end{proof}

The consequence of the previews lemma is that if the $\prod_{c\in(f')^{-1}(0)}(f^n(c)-\alpha)$ is a square, than the Galois group is not isomorphic to $[S_d]^n$. Together with Lemma~\ref{lemma:discriminant}, we can show a well-known result about post-critical finite maps. Since the discriminant of $f^n$ is recurrently defined by multiplying points on the orbit of all critical points, the discriminant of $f^n$ will be a perfect square eventually for all large enough $n$, and it is the following lemma.
\begin{lemma}\label{lemma:periodicdiscriminant}
Let $f$ be a PCF polynomial over $K$, and $\alpha\in K$. Let $L$ be the minimum integer such that $f^L(\mathcal{C}_f)$ is a periodic set, and let $O$ be the minimum integer such that $f^{L+O}(\mathcal{C}_f)=f^L(\mathcal{C}_f)$. Then 
\begin{enumerate}
    \item Suppose all critical points are periodic, and the degree of $f$ is odd. Then $\disc(f^{2O}-\alpha)$ is a perfect square in $K$. 
    \item Suppose all critical points are periodic, and the degree of $f$ is even. If the leading coefficient of $f$ is a perfect square or $O$ is even, then $\disc(f^{O+1}-\alpha)$ is a perfect square in $K_f^1(\alpha)$. Otherwise $\disc(f^{O+2}-\alpha)$ is a perfect square in $K_f^{O+1}(\alpha)$.
    \item Suppose a critical point is strictly periodic, and the degree of $f$ is odd. Then $\disc(f^{L+2O-1}-\alpha)$ is a perfect square in $K_f^{L-1}(\alpha)$. In particular $K_f^{L-1}(\alpha)= K$ if $L=1$.
    \item Suppose a critical point is strictly periodic, and the degree of $f$ is even. If the leading coefficient $a_f$ of $f$ is a perfect square or $O$ is even, then  $\disc(F^{L+O}-\alpha)$ is a perfect square in $K_f^{L}$. If $a_f$ is not perfect square and $O$ is odd, then $\disc(f^{L+O+1}-\alpha)$ is a perfect square in $K_f^{L+O}$.
\end{enumerate}
\end{lemma}
\begin{proof}
Let us use an ambiguous terminology, a potential non-perfect square part of a product $a_1^{e_1}\cdots a_n^{e_n}$, to be $a_1^{e_1'}\cdots a_n^{e_n'}$ where $e_i'=0$ for even $e_i$ and $e_i'=1$ for odd $e_i$. Note that a potential non-perfect square part of an element $a$ in $K$ is not unique. It is up to how we express $a$ in terms of products. Obviously if a potential non-perfect square part of an element $a$ is a perfect square, then $a$ is a perfect square.

We claim that the potential non-perfect square part of the product~\ref{eq:disc} of $\disc(f^n-\alpha)$ over $K$ is
\begin{equation}\label{odd}
\left(\left(\dfrac{d}{4}\right)d\right)^{b_n}\left(\prod_{k=1}^n\prod_{i=1}^{d-1}f^k(c_i)-\alpha\right)
\end{equation}
for odd $d$, and is
\begin{equation}
a_f^{b_n}\left(\prod_{i=1}^{d-1}f^n(c_i)-\alpha\right)
\end{equation}\label{even}
for even $d$, where $b_n=0$ if $n$ is even, and $b_n=1$ otherwise. We will prove this claim in the end, but let us first see what happens if this assertion is true. We will show it case by case.
\begin{enumerate}
    \item We use Equation~\ref{odd}. It is easy to see that the potential non-perfect square part is a perfect square by plugging $2O$ for $n$.
    \item If $a_f$ is a perfect square, then the coefficient of Equation~\ref{even} can be ignored. If $O$ is even, then both $1$ and $1+O$ are odd. Then we can see that $\prod f^{O+1}(c_i)-\alpha=\prod f(c_i)-\alpha$ which implies the potential non-perfect square part of $\disc(f^{O+1}-\alpha)$ is a perfect square in $K_f^1(\alpha)$. Note that the square root of $a_f(\prod f(c_i)-\alpha)$ is in $K_f^1(\alpha)$. Thus if $O$ is odd, then we will find the square root of $\prod f(c_i)-\alpha$ in $K_f^{O+1}(\alpha)$ which implies the square root of $a_f$ is also in $K_f^{O+1}(\alpha)$. Then it is clear that $\disc(f^{O+2}-\alpha)$ is a perfect square in $K_f^{O+1}(\alpha)$ by observing Equation~\ref{even}.
    \item Note that $f^L(\mathcal{C}_f),\ldots, f^{L+2O-1}(\mathcal{C}_f)$ repeat twice. Thus the product
    \[
    \prod_{k=L}^{L+2O-1}\prod_{c\in\mathcal{C}_f}f^k(c)-\alpha
    \]
    is a perfect square. Moreover both $L-1$ and $L+2O-1$ are either even or odd, so we observe Equation~\ref{odd} by plugging $L+2O-1$ for $n$ and conclude that $\disc(f^{L+2O-1}-\alpha)$ is a perfect square in $K_f^{L-1}(\alpha)$.
    \item We will only proof the case where $a_f$ is a perfect square. Other cases are similar, and we left the proof for readers. Since $a_f$ is a perfect square, we can ignore $a_f$ in Equation~\ref{even}. Then we notice that $\prod_{c}f^L(c)-\alpha=\prod_{c}f^{L+O}(c)-\alpha$, so $\disc(f^{L+O}-\alpha)$ is a perfect square in $K_f^{L}(\alpha)$.
\end{enumerate}

Now let us clean up the remaining assertion. We will show it by induction. For an odd $d$ the initial case is
\[
\disc(f-\alpha)=\left(\dfrac{d}{4}\right)a_f^{d-1}d^d\prod_{i=1}^{d-1}f(c)-\alpha
\]
and easy observation shows the initial case hold.

Now suppose $n$ is true for all $n<N$. Easy observation shows that the non-perfect square part of
\begin{align*}
    \disc(f^N-\alpha) & =\left(\dfrac{d}{4}\right)a_f^{d^{2N-1}-1}d^{d^N}(\disc(f^{N-1}-\alpha))^d\prod_{i=1}^{d-1}f(c)-\alpha
\end{align*}
is
\[
\left(\dfrac{d}{4}\right)d\disc(f^{N-1}-\alpha)\prod_{i=1}^{d-1}f^N(c)-\alpha=\left(\dfrac{d}{4}\right)d\left[\left(\left(\dfrac{d}{4}\right)d\right)^{b_{N-1}}\prod_{k=1}^{N-1}\prod_{i=1}^{d-1}f^k(c_i)-\alpha\right]\prod_{i=1}^{d-1}f^N(c)-\alpha.
\]
Thus the assertion is true for odd $d$. Similar observation will show that the assertion is true for even $d$.
\end{proof}

\section{Main results}\label{sec:3}
\subsection{Universal embedding of dynamical groups of PCF maps}
Intuitively, the dynamical Galois group loses all quadratic extensions given in the above Lemma. Following this idea, we can show the following well-known result.
\begin{corollary}
If $f$ is an odd degree PCF polynomial, then there exists some integer $N$ such that $\Gal_f^N(\alpha)$ is not isomorphic to $[S_d]^N$. Moreover, the index of the arboreal representative of $\Gal_f(\alpha)$ in $[S_d]^\infty$ is infinite.
\end{corollary}
\begin{proof}
This directly deduces from Theorem~\ref{main theorem 2} and Proposition~\ref{prop:unboundedindex}.
\end{proof}

Given an integer $m$, I will need to construct two subgroup of $\Aut(T_n)$, denoted by $E_n^m$ and $F_n^m$, where $T_n$ is a $d$-ary tree to $n$ levels, and show that $|\Aut(T_n)|/|E_n^m|$ and $|\Aut(T_n)|/|E_n^m|$ goes to infinity as $n\to\infty$. In the end, I use a lemma to show that, for some fixed integer $m$ depending on $f$, $\Gal_f^n(\alpha)$ is a subgroup of $E_n^m$ or $E_n^m$ for all $n$.

\begin{theorem}\label{main theorem 2}
Let $f$ be a PCF polynomial over a number field $K$, and suppose $\alpha$ is not periodic. Let $\mathcal{C}_f$ be the set of all critical points of $f$. Let $L$ be the minimum integer such that $f^L(\mathcal{C}_f)$ is a periodic set, and let $O$ be the minimum positive integer such that $f^{L+O}(\mathcal{C}_f)=f^{L}(\mathcal{C}_f)$.
\begin{enumerate}
    \item Suppose the degree of $f$ is odd, and $L\leq 1$. Then $\Gal_f^n(\alpha)$ is a subgroup of $E_n^{2O}$.
    \item Otherwise $\Gal_f^n(\alpha)$ is a subgroup of $F_n^{(m_1,m_2)}$ where
    \[
    (m_1,m_2)=\begin{cases}
    (O+1,1),\quad L=0\text{ and }O\text{ is even;}\\
    (O+2,2),\quad L=0\text{ and }O\text{ is odd;}\\
    (L+2O-1,L-1),\quad L>1\text{ and }d\text{ is odd;}\\
    (L+O,L),\quad \quad L>1\text{, $d$ is odd, and $a_f$ is a perfect square or $O$ are even;}\\
    (L+O+1,L+O),\quad L>1\text{, $d$ is odd and $a_f$ is not a perfect square and $O$ is odd.}
    \end{cases}
    \]
\end{enumerate}
\end{theorem}
\begin{proof}
Since $\alpha$ is not periodic, the dynamical tree of $f$ at $\alpha$ is graphically isomorphic to a full $d$-ary tree where $d$ is the degree of $f$.

Let us assume $L\leq 1$ and the degree of $f$ is odd. By Lemma~\ref{lemma:periodicdiscriminant} the $\disc(f^{2O}-\alpha)$ is a perfect square in $K$, so $\Gal_f^{m}(\alpha)$ must be isomorphic to a subgroup of $\Aut(T_m)\cap \ker(\sgn_m)$. For $n> m$ we use induction. Let $f^{-1}(\alpha)=\{y_1,\ldots, y_d\}$. By the induction hypothesis  the dynamical Galois group of $f$ at $y_i$, $\Gal(K_f^{n}(y_i)/K(y_i))$ is isomorphic to a subgroup of $E_n^m$. Clearly $\Gal(K_f^{n+1}(\alpha)/K)$ is isomorphic to a subgroup of $E_n^m\wr \Aut(T_1)$. Using the fact that the discriminant of $f^{m}-\alpha$ is a perfect square in $K$, it shows that any permutation acting on the $n+1$-th level of the tree $T_n$ is even when the permutation restricted to the $m$-th level. Thus $\Gal_f^{n+1}(\alpha)$ is in the kernel of $\sgn_m$.

For the rest of cases Lemma~\ref{lemma:periodicdiscriminant} implies there exists an integer $m_1$ such that $\disc(f^{m_1}-\alpha)$ is a perfect square in $K_f^{m_2}(\alpha)$. Thus $\sigma\in \Gal(K_f^{m_1}(\alpha)/K_f^{m_2}(\alpha))$ is an even permutation. To see it more clear, let $f^{-m_1}(\alpha)=\{y_i\mid i\in I(T_{m_1})\}$, the Galois group $\Gal(K_f^{m_1}(\alpha)/K_f^{m_2}(\alpha))$ embeds into a subgroup of the 
\[
\prod_{i\in I(T_{m_1})}\Gal(K_f^{m_1-m_2}(y_i)/K_f^{m_2}(\alpha)),
\]
The $\Gal(K_f^{m_2}/K)$ acts on the base point $y_i$, so $\Gal(K_f^{m_1}(\alpha)/K)$ can be embedded, denote the map by $\tau$, into $\Aut(T_{m_1-m_2})\wr\Aut\Aut(T_{m_2})$. Since $\sigma$ is an even permutation, we have $\tau(\sigma)\in \ker(\sgn^{(m_1,m_2)})$. Similar to the proof of periodic cases, we can use induction to show that $\Gal_f^n(\alpha)$ can be embedded into $F_n^{(m_1,m_2)}$. 
\end{proof}

\subsection{The rank of $E_n^2$}\label{sec:the rank of E_n}
Let $S$ be a subset of a group $G$. If $G=\langle S\rangle$, we call $S$ a generating set of $G$. It is a long-term goal to determine the cardinality of minimal generating set of $G$. Let
\[
d(G)=\min\{\# S\mid G=\la S\ra\}.
\]
Here we recall an important theorem regarding $d(G)$ by Dalla Volta and Lucchini (see~\cite{DallaVoltaLucchini}).
\begin{lemma}\label{lemma:boundedrank}
If a finite non-cyclic group $G$ contains a unique minimal normal
subgroup $M$, then we have
\[
d(G)\leq\max\{2,d(G/M)\}.
\]
\end{lemma}
An easy consequence of this lemma is that if there is a tower of normal subgroup, say
\[
\{1\}\triangleleft N_1\triangleleft N_2\triangleleft\cdots\triangleleft N_k\triangleleft G \label{eq:towerofnormalsubgroup} \tag{\P}
\]
such that for $N\triangleleft G$ we have either $N\triangleleft N_k$ or $N=N_i$ for some $i$. Then $d(G)\leq \max\{2,d(G/N_i)\}$ for all $i$.

In the following we want to show that $E_n^m$ has a tower of normal subgroups satisfying Condition~\ref{Condition}. We will argue a proper normal subgroup $N\triangleleft G$ is the unique one by showing:
\begin{enumerate}
    \item For any proper normal subgroup $N'$ of $G$ we have $N\cap N'\neq\emptyset$.
    \item $N$ is a smallest normal subgroup.
\end{enumerate}
\begin{lemma}\label{lemma:scsneqc}
For a nontrivial element $\sigma\in E_n^m$ there exists an element $\mathbf{c}\in \ker(\res_{n-1}:E_{n}^m\to E_{n-1}^m)$ such that 
\[
\sigma \mathbf{c}\sigma^{-1}\mathbf{c}^{-1}\neq 1.
\]
For a nontrivial $\mathbf{c}\in \ker(\res_{n-1}:E_{n}^m\to E_{n-1}^m)$ there exists an element $\sigma\in E_n^m$ such that 
\[
\sigma \mathbf{c}\sigma^{-1}\mathbf{c}^{-1}\in M\setminus\{1\}.
\]
\end{lemma}
\begin{proof}
Note that for any $\sigma=((a_i)_{i\in I};b)\in E_n^m$, let $\mathbf{a}=((a_i)_{i\in I};1)$ and $\mathbf{b}=((1)_{i\in I}, b)$, $\sigma=\mathbf{b}\mathbf{a}$.

Given $\sigma\in E_n^m$, we can write $\sigma\mathbf{c}\sigma^{-1}$ as
\[
\mathbf{b}\mathbf{a}\mathbf{c}\mathbf{a}^{-1}\mathbf{b}^{-1}.
\]
If $\mathbf{a}$ is identity, we only need to choose $\mathbf{c}$ such that $c_i\neq c_j$ where $\mathbf{b}^{-1}(i)=j$. If $\mathbf{a}$ is not identity, then for $d\geq 5$ we can choose $c_i$ and $c_j$ that is products of different length of disjoint cycles, i.e. a product of two disjoint $2$-cycles and a 3-cycle, note that this choice makes $\mathbf{c}\in \ker(\res_{n-1})$. For $d=3$ we may choose $c_1$ be a 3-cycle, $c_2$ and $c_3$ a $2$-cycle. Then $a_ic_ia_i^{-1}\neq a_jc_ja_j^{-1}$ by how we choose the permutation. Then $\mathbf{b}(\mathbf{a}\mathbf{c}\mathbf{a}^{-1})\mathbf{b}$ is back to the case that we assume $\mathbf{a}$ is identity. Hence the result follows.  

Given $\mathbf{c}=((c_i);1)\in \ker(\res_{n-1}:E_n^m\to E_{n-1}^m)$ and for $\sigma\in E_n^m$, we can write $\sigma\mathbf{c}\sigma^{-1}$ as
\[
\mathbf{b}\mathbf{a}\mathbf{c}\mathbf{a}^{-1}\mathbf{b}^{-1}.
\]
If $c_i=c_j$ for all $i$ and $j$, then we may take two transpose $a_1$ and $a_2$ that are not disjoint from $c_1$ and $c_2$. By taking $\mathbf{b}$ to be the identity, we again show the desired result. 

Now for some distinct indices $i$ and $j$ we have $c_i\neq c_j$, then by Lemma~\ref{lemma:transitive} we have $\mathbf{b}$ transpose $i$ and $j$. For both $c_i$ and $c_j$ even or odd we take $\mathbf{a}$ to be the identity, and it implies $\sigma\mathbf{c}\sigma^{-1}\mathbf{c}\in M\setminus\{1\}$. Now without lose of generality we may assume $c_i$ is an even permutation and $c_j$ is an odd permutation. Since $\mathbf{c}\in E_n^n$, we must has an other odd permutation $c_{j'}$ with $j\neq j'$. If $c_j=c_{j'}$, we may process a conjugation on $c_{j'}$ by some even permutation such that $c_j\neq c_{j'}$. Thus we may assume $c_j\neq c_{j'}$. By Lemma~\ref{lemma:transitive} we have $\mathbf{b}$ transposing $j$ and $j'$. Following the above processes, we find a $\sigma$ such that the desired result holds.
\end{proof}
\begin{proposition}\label{propostion:inverseautomorphismisbyconjugate}
Let $\sigma\in [S_d]^n$. There exists an $\tau\in [S_d]^n$ such that $\sigma^\tau=\sigma^{-1}$.
\end{proposition}
\begin{proof}
We show the statement by induction. For $n=1$ the statement trivially follows from the fact that $\Aut(S_d)\cong S_d$. 

For the sake of induction argument we assume the statement is true for a positive integer $N$. Given an element $\sigma\in \Aut(T_{N+1})\cong\Aut(T_1)\wr\Aut(T_{N})$ we may write $\sigma$ as
\[
\sigma=((\sigma_i)_{i\in I(T_N)},\sigma_0)
\]
for $\sigma_i\in S_d$. Thus for each $i$ we have $\tau_i$ such that $\sigma_i^{\tau_i}=\sigma_i^{-1}$ by the fact that $\Aut(S_d)\cong S_d$ and the induction hypothesis. Now we observe
\[
((\tau_{\tau_0^{-1}(i)})_{i\in I(T_d)};1)(1;\tau_0)((\sigma_i)_{i\in I(T_d)};\sigma_0)(1;\tau_0)^{-1}((\tau_{\tau_0^{-1}(i)})_{i\in I(T_d)};1)^{-1}
\]
is the inverse of $\sigma$.
\end{proof}
\begin{lemma}\label{lemma:elementinres}
Suppose one of the followings holds:
\begin{enumerate}
    \item Let $G$ be a group. Assume that, for any $g\in G$, there exists $h$ such that $g^h\neq g^{-1}$.
    \item Let $\res_n:G\wr\Aut(T_n)\to \Aut(T_n)$. Assume $H\subseteq \ker(\res_n)$ is a normal subgroup of $G\wr\Aut(T_n)$, and for any $h\in H$ there exists $\sigma\in\Aut(T_n)$ such that $\sigma(h)\neq h^{-1}$.
\end{enumerate}
 Then any proper normal subgroup of $G\wr\Aut(T_n)$ or $H\ltimes\Aut(T_n)$ respectively contains a nontrivial element in $\{((g_1,\ldots,g_n);1)\mid g_i\in G\}$
\end{lemma}
\begin{proof}
For any element $\tau\in N\triangleleft G\wr\Aut(T_n)$ there exist $\mathbf{g}=((g_i)_{i\in I(T_n)};1)$ and $\mathbf{s}=(1;\sigma)$ for $g_i\in G$ and $\sigma\in \Aut(T_n)$ such that $\tau=\mathbf{s}\mathbf{g}$. By Proposition~\ref{propostion:inverseautomorphismisbyconjugate} there exists $\mathbf{t}\in G\wr\Aut(T_n)$ such that $\mathbf{t}\mathbf{s}\mathbf{t}^{-1}=\mathbf{s}^{-1}$. Thus we have
\[
\tau^{\mathbf{t}}\tau=((\ast);1).
\]
For the first statement one can conjugate some $g_i$ by a proper $h$ such that $g_i^h\neq g_i^{-1}$. Thus we may conjugate $\tau$ by some proper element $\mathbf{h}\in \ker(\res_n)$ such that $\tau^{\mathbf{t}}\tau^{\mathbf{h}}\neq 1$.

For the second argument since $\Aut(T_n)$ acts on the index of $G^{d^n}$ and $H$ is normal, the natural map $\phi:\Aut(T_n)\to \Aut(H)$ is an homomorphism which implies $H\ltimes\Aut(T_n)$ is well-defined. By the assumption we can again construct an nontrivial element in $\ker(\res_n)$ via a similar manner.
\end{proof}
Let $e_i\in C_2^{d^{n-1}}$ be an element of the form $(1,\ldots,-1,\ldots,1)$ where $-1$ only appears on the $i$-th coordinate. Let $k_i=d^{i-1}+1$. We define
\[
X_i=\{(e_1+e_{k_i};1)^{\sigma}\mid \sigma\in \Aut(T_n)\}\subseteq C_2^{d^n}\ltimes \Aut(T_{n})
\]
for $i=1,2,\ldots, n$. We have the following important lemma.
\begin{lemma}\label{lemma:C2AutTn}
Let $H_i=\la X_i\ra\ltimes\Aut(T_n)$, and let $\res_n:H\to\Aut(T_n)$. Then $\la X_1\ra$ is the unique minimal normal subgroup of $H_i$. In particular, let $H=\la X_1\ra\ltimes \Aut(T_1)$, we have $\la X_1\ra=\ker(\res_1)$.
\end{lemma}
\begin{proof}
To show $H_1$ is the unique minimum proper subgroups of $H_n$ we apply the two-step argument. Firstly we notice that the order of the generators are $2$. Secondly $H_n$ is a subgroup of $C_2\wr\Aut(T_n)$, and for any $h\in \la X_i\ra$ one can find $\sigma\Aut(T_n)$ such that $\sigma(h)\neq h$. By Lemma~\ref{lemma:elementinres} any proper normal subgroup of $H_n$ intersecting with the kernel of $\res:H_n\to\Aut(T_n)$ contains a nontrivial element. 

We will use the nontrivial element to construct $e_1+e_2$ by a sequence of conjugation. Since $\Aut(T_{n})$ acts transitively on $I(T_n)$, we can assume that the nontrivial element $h$ exists in the intersection has $-1$, the nontrivial element of $C_2$, in the block of $\mathcal{B}(T_n/T_{n-1})(1,\ldots, 1)$. We also notice that $\sigma=(((12\cdots d),1,\ldots,1);1)$ shift the block by conjugation, so $h^\sigma h$ has even $-1$ on the first block and $1$ on all entries of the reside blocks. Next we claim that for even $-1$ and odd $1$ on a black with $S_d$ acting on the index, we can always have
\[
(-1,\ldots, -1,1,\ldots, 1)
\]
by conjugating an even permutation. If $h$ is already in the desired form, then we are done. Otherwise there are three unceasing index, say $i$, $i+1$ and $i+2$, with $1$ and $-1$ distributing in one of the following layouts
\begin{align*}
    1, -1, -1\\
    -1, 1, -1\\
    1, -1, 1\\
    1, 1, -1.
\end{align*}
One can applying conjugation by the 3-cycle $(i, i+1, i+2)$ several times to move the above layouts to either $-1,-1,1$ or $-1,1,1$. If the element after applying conjugation is not in the desired form, we repeat the process. Since there are only finite many indices, we can have the desired from after finite many times of applying the process. Moreover since we apply $3-cycle$ each time, the desired result can be achieved by applying an even permutation. Now we can conjugate $(-1,-1,\ldots, -1,1,\ldots,1)$ by $(12\cdots d)$ to get $(1,-1,\ldots, -1,1,\ldots, 1)$. Note that the multiplication of $(-1,\ldots,-1,1,\ldots,1)$ and $(1,-1,,\ldots,-1,1,\ldots,1)$ gives $(-1,1,\ldots,1,-1,1,\ldots,1)$. Finally we repeating the second process to get $(-1,-1,1,\ldots, 1)$. Since an arbitrary proper normal subgroup contains the generator of $H_1$, we shows that $\la X_1\ra$ is the unique minimal normal subgroup of $H_i$.
\end{proof}
\begin{remark}
The above proof can be written as the following psudocode.

\begin{algorithm}[H]
\KwData{A tuple $\mathbf{a}=(a_1,\ldots, a_d)\in C_2^d$ with $\prod{a_i}=1$ and odd $d$}
\KwResult{The tuple $(-1,-1,1,\ldots,1)$}
\BlankLine
\While{$\mathbf{a}\neq(-1,-1,1,\ldots1)$}{
    $\mathbf{a}'=\mathbf{a}$\;
    $\mathbf{b}=(-1,\cdots,-1,1,\cdots,1)$  with same number of $-1$ as in $\mathbf{a}$\;
    \While{$\mathbf{a}'\neq\mathbf{b}$}{
        Find the smallest index $i$ such that $a'_i=1$ and $a'_{i+1}=-1$\;
        $(a'_i,a'_{i+1},a'_{i+2})=(a'_{i+1},a'_{i+2},a'_{i})$\;
    }
    $\mathbf{a}'=(a_d',a_1',a_2',\ldots,a_{d-1}')$\;
    $\mathbf{a}=\mathbf{a}\mathbf{a}'$\;
}
\caption{The pseudocode of the above proof}
\end{algorithm}
\end{remark}
\begin{proposition}\label{proposition:a1a2a3a4}
Let $a_1=(123)$ or $(12\cdots d)$. Then for $\mathbf{a}=((a_1,a_2,a_3,\ldots, a_d);1)\in E_2^2$ there exists $\sigma\in E_2^2$ such that $\mathbf{a}^{\sigma}=((a_1,a_2^{-1},a_3^{-1},\cdots,a_d^{-1});1)$ or $((a_1,a_2,a_3^{-1},\cdots,a_d^{-1});1)$.
\end{proposition}
\begin{proof}
For any $a_i$ there exists an element $s_i\in S_d$ such that $a_i^{s_i}=a_i^{-1}$. If $((1,s_2,\ldots, s_d);1)$ is in $E_2^2$, then it is done. Otherwise there are odd many odd permutations $s_i$. Without lose of generality we assume $s_2$ is odd. Then we can conjugate $\mathbf{a}$ by $((1,1,s_3,\ldots, s_{d});1)\in E_2^2$ to get $((a_1,a_2,a_3^{-1},\cdots,a_d^{-1});1)$. 
\end{proof}

\begin{lemma}\label{lemma:E_n->E_n-1}
Let $d$ be an odd integer, and let $\res_{n}:E^2_{n}\to E^2_{n-1}$ be the natural restriction from $\Aut(T_{n+1})$ to $\Aut(T_n)$. Let $M_n=((A_d)_{i\in I},1)\subseteq E_{n+1}\subseteq \Aut(T_1)\wr\Aut(T_n)$ where $A_d$ is the alternative group of degree $d$. Then $E_{n}^2$ has a tower of normal subgroup
\[
\{1\}\triangleleft M_n\triangleleft\ker(\res_{n})\triangleleft E_{n}^2
\]
satisfying Condition~\ref{eq:towerofnormalsubgroup}.
\end{lemma}
\begin{proof}
We first need to show that $M_n$ is a normal subgroup. It is clear a subgroup of $E_{n}^2$. For an element $\sigma\in E_{n}^2$ one can express $\sigma$ as $((a_i)_{i\in I},b)\in\Aut(T_1)\wr\Aut(T_{n})$. Let $\mathbf{a}=((a_i)_{i\in I},1)$ and $\mathbf{b}=((1),b)$. Then we have $\sigma=\mathbf{b}\mathbf{a}$ and
\[
\sigma M_n\sigma^{-1}=\mathbf{b}\mathbf{a}M_n\mathbf{a}^{-1}\mathbf{b}^{-1}=\mathbf{b}M_n\mathbf{b}^{-1}=M_n.
\]
Since we show $M_n$ is a subgroup of $E_n^2$ and a normal subgroup of a group containing $E_n^2$, $M_n$ is a normal subgroup of $E_n^2$.

Next we shall show that a nontrivial normal subgroup $N$ contains an element in $M_n$. Given an element $\sigma\in N$ and $\mathbf{c}=((c_i)_{i\in I},1)\in E^n_{n}$ we have $\sigma\mathbf{c}\sigma^{-1}\mathbf{c}^{-1}\in N$. If $\sigma$ is in $M_n$, then it is done. Now we can assume $\sigma\not\in M_n$, and Lemma~\ref{lemma:scsneqc} shows that for a given $\sigma$ we have $\mathbf{c}\in \ker(\res_{n-1})$ such that
\[
\sigma\mathbf{c}\sigma^{-1}\mathbf{c}^{-1}\neq 1.
\]
Since $c$ is also in a normal subgroup, it follows that $1\neq \sigma\mathbf{c}\sigma^{-1}\mathbf{c}^{-1}\in N\cap \ker(\res_n)$. Denote $\sigma\mathbf{c}\sigma^{-1}\mathbf{c}^{-1}$ by $\gamma$. Consequently if $\gamma$ has order greater than $2$, then we have $\gamma^2\in N\cap M_n$. Otherwise using the second part of Lemma~\ref{lemma:scsneqc}, we have some $\delta\in E_n^2$ such that $\delta\gamma\delta^{-1}\gamma^{-1}\in N\cap M_n$.

If we can further show that $M_n$ is a smallest normal subgroup of $E^2_{n}$, then we can conclude that $M$ is the unique normal subgroup of $E^2_{n}$. We will use induction to show that $M_n$ is the smallest normal subgroup of $E^2_n$ for $n\geq 2$. Let
\begin{align*}
    \rho_i:\ker(\res_1:E_2^2\to E_1^2) &\to \Aut(T_1);\\
           ((a_i)_{i\in I},1) &\mapsto a_i.
\end{align*}
be a homomorphism. If $N$ is a normal subgroup of $E_2^2$, then $M_2\cap N$ is a normal subgroup of $\ker(\res_1)$ which implies $\rho_i(M_2\cap N)$ is a normal subgroup of $\Aut(T_1)$. Thus $\rho_i(M_2\cap N)\cong \{1\}$ and $A_d$. Since $\res_{1}(E_2^2)=\Aut(T_1)$ acts transitively on the level $2$ of the tree, we conclude that $\rho_i(M_2\cap\ker(\res_1))\cong A_d$ for all $i$. Without lose of generality we may take $i=1$, so for any $a_1\in A_d$ and for $a_i\in S_d$ with $i=2,\ldots, d$ the element $\mathrm{a}=((a_1,a_2,\ldots, a_{d});1)$ exists in $N\cap M_2$. In particular we may choose $a_1=(123)$ or $(12\cdots n)$, where $\{(123),(12\cdots n)\}$ forms a generating set of $A_d$ for an odd $d$. We claim $\mathrm{a}'=((a_1,1,\ldots, 1);1)$ exists in the intersection. By Proposition~\ref{proposition:a1a2a3a4} we know we can conjugate $\mathrm{a}$ to $\mathrm{a}_1=((a_1,a_2^{-1},\ldots, a_d^{-1});1)$ or $\mathrm{a}_2=((a_1,a_2,a_3^{-1},\ldots, a_d^{-1});1)$. If $\mathrm{a}$ conjugates to $\mathrm{a}_1$, then a power of $\mathrm{a}\mathrm{a}_1$ is in the desired form. If the result of conjugation is the second case, then we can conjugate $\mathrm{a}\mathrm{a}_2$ by $((1,s_2,(12),1,\ldots, 1);1)\in E_2^2$ to get $\mathrm{a}'=((a_1^2,a_2^{-2},1,\ldots,1);1)$, then the desired result can be achieve by taking a power of $\mathrm{a}\mathrm{a}_2\mathrm{a}'$. Therefore we prove that $M_2$ is a proper minimum subgroup of $E_2^2$. Let us assume $M_n$ is a proper minimum subgroup of $E_n^2$. It implies $\mathcal{A}=\{((a,1,\ldots,1);1)\in\Aut(T_1)\wr\Aut(T_{n-1})\mid a\in A_d\}$ is a subgroup of $M_n\subseteq E_n^2$, and so
\[
\mathcal{A}'=\{((a',1,\ldots,1);1)\in E_n^2\wr\Aut(T_1))\mid a\in \mathcal{A}\}=\{((a,1,\ldots,1);1)\in \Aut(T_1)\wr\Aut(T_n))\mid a\in A_d\}
\]
is a subgroup of $M_{n+1}$. By Lemma~\ref{lemma:transitive} it follows that $M_{n+1}$ is a minimum subgroup of $E_{n+1}^2$. Thus we show that $M_{n}$ is the unique proper minimum subgroup of $E_n^2$ for $n\geq 2$.

We claim that $E_n^2/M_n$ is isomorphic to $C_2\wr\Aut(T_{n-1})$. Let us define a homomorphism
\begin{align*}
    \phi:&E_n^2\subseteq\Aut(T_1)\wr\Aut(T_{n-1})\to C_2\wr\Aut(T_{n-1})\\
    &((a_i)_{i\in I(T_{n-1})},b)\mapsto ((\sgn(a_i))_{i\in I(T_{n-1})},b).
\end{align*}
It is standard to check $\ker(\phi)\cong M_n$ and $\phi$ is onto. Finally we apply Lemma~\ref{lemma:C2AutTn} to conclude that $\ker(\res_{n-1})/M_2$ is the unique minimal subgroup of $E_n^2/M$. Now we show every assertion of this lemma. 
\end{proof}
\begin{theorem}
The rank of the group $E_n^2$ is $2$ for all $n$. Moreover $E_n^2$ has a unique chief series. 
\end{theorem}
\begin{proof}
By Lemma~\ref{lemma:boundedrank} and Lemma~\ref{lemma:E_n->E_n-1} we have $d(E_n^2)\leq \max\{2,d(E_n^2/M_n)\}$ and $d(E_n^2/M_n)\leq \max\{2,d(E_n^2/\res_n)\}$, so $d(E_n^2)\leq \max\{2,E_n^2/\res_n\}$. Moreover we have $E_n^2/\res_n\cong E_{n-1}^2$, so we can conclude that 
\[
d(E_n^2)\leq\max\{2,d(E_{n-1}^2)\}\leq \cdots\leq\max\{2, d(E_1^2)\}=2.
\]
Lemma~\ref{lemma:E_n->E_n-1} also shows that $M_{n-1}$ is the proper unique minimal normal subgroup of $E_{n-1}$, so $\res_{n-1}^{-1}(M_{n-1})$ is a unique minimum normal subgroup in $E_n$ that contains the tower of normal subgroup $M_n\triangleleft \ker(\res_n)$. Thus we can recursively construct a unique normal series such that each factor is a chief factor.
\end{proof}
\subsection{Examples}
In this section we are going to investigate the arboreal representatives of $f(z)=2z^3-3z^{2}+1$ at a proper chosen base point $x$. The idea of the proof is almost identical to \cite{B-F-H-J-Y-2017}.

Let $K$ be a number field and let $x\in K\setminus\{0,1\}$. We assume the pair $(K,x)$ satisfies the following condition:
\begin{equation}\label{Condition}
\begin{array}{c}
     \text{there exists $\mathrm{p}$ and $\mathrm{q}$ of $K$ lying above $2$ and $3$ such that}  \\
     \text{either $v_\mathrm{q}(x)=1$ or $v_\mathrm{q}(1-x)=1$, and either $v_\mathrm{p}(1-x)= 1$ or $v_\mathrm{p}(x)=1$}
\end{array}
\end{equation}
\begin{lemma}
If $(K,x)$ satisfies Condition~\ref{Condition}, then $f^n(z)-x$ is Eisenstein at $\mathrm{q}$ for all $n\geq 1$. In particular $f^n(z)-x$ is irreducible.
\end{lemma}
\begin{proof}
For $n=1$ and $v_\mathrm{q}(x)=1$ we have
\[
f(z)\equiv (-z^3+1)\equiv -(z-1)^3\mod \mathrm{q},
\]
so $f(z-1)-x$ is Eisenstein. For $n=2$ we have
\[
f^2(z)=-(f(z)-1)^3\equiv -(-z^3)^3=z^{9}\mod \mathrm{q}.
\]
Then it is easy to prove the statement by induction.

For $v_\mathrm{q}(1-x)=1$ we have
\[
f(z)\equiv(-z^3)\mod\mathrm{q},
\]
and
\[
f^2(z)\equiv -f(z)^3\equiv (z-1)^3\mod \mathrm{q}.
\]
Again the conclusion will follow easily by induction.
\end{proof}
\begin{lemma}\label{lemma:ramification index}
Let $(K,x)$, $\mathrm{p}$ and $\mathrm{q}$ satisfy Condition~\ref{Condition}, and let $y\in f^{-n}(x)$.  Then:
\begin{itemize}
    \item There are prime $\mathrm{p}'$ and $\mathrm{q}'$ of $K(y)$ lying above $\mathrm{p}$ and $\mathrm{q}$ respectively such that
    \[
    e(\mathrm{p}'/\mathrm{p})=2^n\textbf{ and }e(\mathrm{q}'/\mathrm{q})=3^n
    \]
    \item $(K(y),y)$ satisfies Condition~\ref{Condition}.
\end{itemize}
\end{lemma}
\begin{proof}
    We first prove the assertion for $\mathrm{q'}$. Suppose $v_\mathrm{q}(x)=1$, then $y\in f^{-1}(x)$ has $(y-1)^3\mod q$ which implies there exists a unique prime $\mathbf{q}'$ of $K(y)$ above $q$ satisfies $v_{\mathrm{q}'}(1-y)=1$, i.e. $e(\mathrm{q}'/\mathrm{q})=3$. On the other hand, if we have $v_{\mathrm{q}}(1-x)=1$, then it implies there exists a unique prime $\mathrm{q}'$ of $K(y)$ such that $v_{\mathbf{q}'}(y)=1$. 
    
    For $\mathrm{p}'$ we also prove the statement case by case. Let $v_\mathrm{p}(1-x)=1$, then we note that the Newton polygon of $f(z)-x$ has a segment of length $2$ and height $1$ which implies there exists a prime $\mathrm{p}'$ above $\mathrm{p}$ such that $e(\mathrm{p}'/\mathrm{p})=2$. Moreover we have $v_\mathrm{p'}(y)=1$ for $y\in f^{-1}(x)$.
    
    Let $v_\mathrm{p}(x)=1$, then we have
    \[
    f(z)-x\equiv -z^2+1\equiv -(z-1)^2\mod\mathrm{p}.
    \]
    Thus there exists a prime $\mathrm{p}'$ over $\mathrm{p}$ with $v_{\mathrm{p}'}(1-y)=1$.
    
    Since $x$ is arbitrary chosen with an arbitrary $y\in f^{-1}(x)$ such that $(K(y),y)$ satisfies Condition~\ref{Condition}, the statement is true by induction.
\end{proof}

\begin{corollary}
Let $(K,x)$ be as in Proposition~\ref{lemma:ramification index}. Then $6^n|[K_f^n(x):K]$. In particular there exists an order $2$ and order $3$ elements in $\Gal(K_f^n/K_f^{n-1})$
\end{corollary}
\begin{proof}
The proof is easy by induction, and we refer readers to~\cite{B-F-H-J-Y-2017} for more details. The order $2$ and $3$ elements exist because Cauchy Theorem.
\end{proof}

\begin{theorem}
Let $(K,x)$ satisfies Condition~\ref{Condition}. Then for $n\geq 1$ we have
\[
\Gal_f^n(x)\cong E_n.
\]
\end{theorem}
\begin{proof}
In~\ref{Condition} they prove their theorem by notice that the existence of elements of order $2$ and $3$ which implies there is a two 2-cycle permutation in $\Gal_f^2(x)$ and a $3$-cycle in $\Gal_f^2(x)$ fixing $f^{-1}(x)$. Thus their proof works exactly the same in our case.
\end{proof}
\section{Questions}
There are some questions that the author cannot settle in this paper. We choose the polynomial $f(z)=2z^3-3z^2+1$ because the critical points of $f$, $\mathcal{C}_f=\{0,1,\infty\}$, satisfies $f(\mathcal{C}_f)=\mathcal{C}_f$. We can define a more general polynomial with the same critical portrait. Let $B(x)$ be an odd degree Belyi map that has three fixed critical points $0$, $1$ and $\infty$, and then $-B(z)+1$ will transpose $0$ and $1$ and fix $\infty$. By Theorem~\ref{main theorem 2} it is clear that the $n-$th dynamical subgroup of $-B(z)+1$ is a subgroup of $E_n^2$. Hence the question is "do we have an isomorphism?".
\begin{question}
Let $K$ be a number field. Let $f\in K[z]$ be an odd degree polynomial. Assume $\mathcal{C}_f=\{c_1,c_2,c_3\}$ with the critical portrait
\[
c_1\to c_2,\ c_2\to c_1,\ c_3\to c_3. 
\]
Then do we have $\Gal_f^n(\alpha)\cong E_n^2$ with a properly chosen basepoint $\alpha\in K$?
\end{question}
A possible approach is using the method introduced in~\cite{BEK2020}. The first step is to show that the Geometric Galois group $\Gal(B^n(z)-t/\overline{K}(t))$, where $t$ is transcendental over $K$, is isomorphic to $E_n^2$, and then deduce that the arithmetic Galois group is also isomorphic to $E_n^2$. The last step is to study the ramification condition of $K_f^n(\alpha)$ over $K$ to show that with a specified basepoint $\alpha$ we have the desired isomorphism.

My next question is the following.
\begin{question}
Is there a pair of PCF polynomial $f$ over $K$ and a basepoint $\alpha\in K$ such that $\Gal_f^n(\alpha)\cong F_n^{(m_1,m_2)}$?
\end{question}
The first thing we have to solve for this question is that we have to find a proper family of PCF maps that has tail length greater than but not equal to $1$. We didn't see an obvious target that satisfying the requirement. 

Finally we would like to know the chief series of $E_n^m$ for $m\geq 3$ and $F_n^{(m_1,m_2)}$. The difficulty of generalize our argument to other cases is the following observation. Recall Lemma~\ref{lemma:C2AutTn}. Let $H_n=\la X_n\ra\Aut(T_n)$ with $n\geq 1$. Then the quotient group $H_n/\la X_1\ra$ correspondent to the following exact sequence
\[
\{1\}\to\la X_1\ra\to X_n\to H_n\to \Aut(T_n)\to \{1\}.
\]
Notice that $\ker(\res_{n-1}:\Aut(T_n)\to\Aut(T_m))$ acts trivially on $\la X_n\ra/\la X_1\ra$. Thus we have
\[
H_n/\la X_1\ra\cong (\la X_n\ra/\la X_1\ra\times \ker(\res_{n-1}))\ltimes \Aut(T_{n-1}). 
\]
Note that both $\la X_n\ra/\la X_1\ra$ and $\ker(\res_{n-1})$ are normal subgroup and has trivial intersection. Thus we don't have a unique tower of normal subgroup satisfying Condition~\ref{Condition}. Thus our argument doesn't work. For $F_n^{(m_1,m_2)}$ the situation seems more complicate, and some part of our argument depends on the fact that $d$ is odd and we have even many $-1$. My question is that
\begin{question}
Can we count the number of chief series in $E_n^m$ for $m\geq 3$? Is there a way to bound the rank of $E_n^m$ from pure algebraic point of views? Do we have a unique chief series in $F_n^{(m_1,m_2)}$ for small $m_1$ and $m_2$?
\end{question}
\section*{Acknowledgments}
The author would like to thank Dr. Alexander Hulpke for very helpful discuss and comments, and Dr. Andrea Lucchini for useful comments on the rank of groups. The author also want to thank his advisor Dr. Thomas Tucker for various support.
\bibliographystyle{plain}
\bibliography{references}
\end{document}